\documentclass{amsart}

\usepackage{amscd,amsthm,amssymb,amsmath,graphicx,verbatim}
\usepackage{xypic}
\usepackage[dvips]{hyperref}
\usepackage[TS1,OT1,T1]{fontenc}
\usepackage[applemac]{inputenc}
\newtheorem{theorem}{Theorem}[section]
\newtheorem{lemma}[theorem]{Lemma}
\newtheorem{corollary}[theorem]{Corollary}
\newtheorem{proposition}[theorem]{Proposition}
\newtheorem{conjecture}[theorem]{Conjecture}

\theoremstyle{definition}
\newtheorem{definition}[theorem]{Definition}
\newtheorem{example}[theorem]{Example}

\theoremstyle{remark}
\newtheorem{remark}[theorem]{Remark}

\numberwithin{equation}{section}

\begin{document}
\title{Equivalences for noncommutative projective spaces} 
\author{Jorge Vitória}\thanks{Supported by FCT - Portugal, research grant SFRH/BD/28268/2006. Special thanks to Izuru Mori, Yoichi Hattori and Michaela Vancliff for their helpful comments and remarks.}
\address{Mathematics Institute, University of Warwick\\ Coventry, CV4 7AL\\ UK \newline\indent email:  j.n.s.vitoria@warwick.ac.uk}
\begin{abstract}
Following Artin and Zhang's formulation of noncommutative projective geometry, we classify up a family of skew polynomial quadratic algebras up to graded Morita equivalence and their corresponding noncommutative projective spaces up to birational equivalences. We also study their point varieties and provide examples of non-isomorphic noncommutative projective spaces. 
\end{abstract}
\maketitle
\begin{section}{Introduction} 
Consider the following family of quadratic graded algebras, with deg$(X_i)=1$:
\begin{equation}\nonumber
S^n_{\omega} = \mathbb{K}\left\langle X_1,...,X_n\right\rangle/\left\langle X_jX_i-\omega_{ij}X_iX_j,\ i,j\in\left\{1,...,n\right\}\right\rangle,
\end{equation}
where $\mathbb{K}$ is an algebraically closed field of characteristic zero and $\omega_{ij}\in \mathbb{K}^{*}$, for all $i$ and $j$. Note that $\omega_{ij}\omega_{ji}=1$.

Our objects of study will be the spaces  $\mathbb{P}^{n-1}_{\omega}:=Proj(S^n_{\omega})$. These are the noncommutative projective spaces associated to $S^n_{\omega}$, following \cite{AZ}. They are pairs $(tails(S^n_{\omega}),\pi S^n_{\omega})$ where $tails(S^n_{\omega})$ is the quotient category of finitely generated graded modules over $S^n_{\omega}$, $gr(S^n_{\omega})$ by its subcategory of torsion modules and $\pi S^n_{\omega}$ is the projection of $S^n_{\omega}$ in the quotient category. $tails(S^n_{\omega})$ is, therefore, playing the role of coherent sheaves over $\mathbb{P}^{n-1}_{\omega}$ and $\pi S^n_{\omega}$ the role of structure sheaf of $\mathbb{P}^{n-1}_{\omega}$, which we shall denote by $O^n_{\omega}$. We shall also use $Gr(R)$ and $Tails(R)$ for the whole category of graded modules and the respective quotient. Hence, $Tails(S^n_{\omega})$ represents quasi-coherent sheaves over $\mathbb{P}^{n-1}_{\omega}$. This follows Artin and Zhang's (\cite{AZ}) formulation of noncommutative projective geometry. We follow the approach in \cite{AKO} to study these categories and we shall keep the notation therein unless stated otherwise. 

Recall (\cite{AZ},\cite{Mori}) that $\mathbb{P}^{n-1}_{\omega}\cong \mathbb{P}^{n-1}_{\omega'}$ if there is an equivalence of categories $F$ between $tails(S^n_{\omega})$ and $tails(S^n_{\omega'})$ such that $F(O^n_{\omega})=O^n_{\omega'}$. An equivalence at the level of graded modules naturally induces an isomorphism of noncommutative projective spaces and thus our aim in section 2 is to study these equivalences

Note that, since the rings $S^n_\omega$ are Noetherian domains, we may localise and hence form a division ring of fractions. We say that two noncommutative projective spaces $\mathbb{P}^{n-1}_\omega$ and $\mathbb{P}^{n-1}_{\omega'}$ are birationally equivalent if the degree zero of the corresponding division rings of fractions of $S^n_\omega$ and $S^n_{\omega'}$ are isomorphic. We study birational equivalences between these spaces in section 3.

An interesting invariant of such a space is its point variety. We recall the definition from \cite{ATV}. We start with point modules.

\begin{definition}[Artin, Tate, Van den Bergh]

A graded $R$-module $M$ is said to be a point module if:
\begin{itemize}
\item $M$ is generated in degree zero;
\item $M_0 = \mathbb{K}$;
\item $dim\ M_i = 1$, $\forall i\geq 0$.
\end{itemize}
\end{definition}

Point modules can then be parametrised, in a natural way, by the point scheme. Given $R$ a connected positively graded AS-regular algebra of dimension $n$ generated in degree one, let $I$ be such that $R\cong T(R_1)/I$ and note that, for each $i\in\mathbb{N}_0$, $I_i$ can be seen as a set of multilinear functions on $(V^*)^i$. Thus we can define projective schemes associated with $I_i$,
\begin{equation}\nonumber
\Omega_i:= \left\{(p_1,...,p_i)\in \mathbb{P}(V^*)^i: f(p_1,...,p_i)=0,\ \forall f\in I_i\right\},
\end{equation}
and clearly we have, for $i\leq j$, a map 
$pr^j_i: \Omega_j\longrightarrow \Omega_i$
which is the restriction of the projection from $\mathbb{P}(V^*)^j$ to $\mathbb{P}(V^*)^i$ on the first $i$ coordinates.
One can check that $\left\{\Omega_i, pr^j_i\right\}$ forms an inverse system of projective schemes (\cite{ATV}).

\begin{definition}[Artin, Tate, Van den Bergh,\cite{ATV}]
The point scheme of $R$ is the inverse limit of the inverse system of projective schemes $\left\{\Omega_i, pr^j_i\right\}$. We refer to the point variety when considering the reduced structure on the point scheme and we  denote it by $\Omega_R$.
\end{definition}

We compute the point variety for $\mathbb{P}^n_{\omega}$ in section 4. This invariant property is then crucial for the examples in section 5 and it is given by the following theorem proved by Mori in \cite{Mori}. Recall that the Gorenstein parameter of a AS-regular $\mathbb{K}$-algebra $R$ of global dimension $n$ is the integer $r$ such that  $\bigoplus\limits_{k\in\mathbb{Z}}{\rm Ext}_{Gr(R)}^n(\mathbb{K},R(k))\cong \mathbb{K}(r)$.

\begin{theorem}[Mori, \cite{Mori}]\label{invariant point}
Let $R$ and $R'$ be graded quotients of quantum polynomial rings of global dimension $n\geq1$ and Gorenstein parameters $r$ and $r'$ in $\mathbb{Z}\setminus\left\{0\right\}$ respectively. If $Proj(R)\cong Proj(R')$ then $\Omega_R \cong \Omega_{R'}$.
\end{theorem}
It is clear that this theorem applies to the algebras $S_{\omega}^n$.

A preliminary useful observation is that no noncommutative projective spaces of different dimensions are isomorphic. We provide an argument for that straight away, as this will ease our notation. 
Recall the following result.
\begin{proposition}[Auroux, Katzarkov, Orlov, \cite{AKO}]
The categories $D^b(tails(S^n_{\omega}))$ possess strong exceptional collections of length $n$, namely $(O_{\omega}^{n},O_{\omega}^{n}(1),...,O_{\omega}^n(n-1))$, where $O_{\omega}^n=\pi S^n_{\omega}$ is as before.
\end{proposition}

Such strong exceptional sequence gives us a derived equivalence between the categories $tails(S^n_{\omega})$ and $mod(B^n_{\omega})$ (\cite{AKO}) where 
\begin{equation}\nonumber
B^n_{\omega} := End_{D^b(tails(S^n_{\omega}))}(\bigoplus_{i=0}^{n-1} O^n_{\omega}(i)). 
\end{equation}
This algebra $B^n_{\omega}$ can be presented as a path algebra with relations, the quiver being the Beilinson quiver (n ordered vertices, n arrows between any two consecutive vertices) and the relations being described as follows. Let $\alpha^k_l$ be the $i$-th arrow starting at vertex $k$ (i.e., representing multiplication by $X_l$, mapping $O^n_{\omega}(k-1)$ to $O^n_{\omega}(k)$), then the ideal of relations is 
\begin{equation}\nonumber
\left\langle \alpha_j^k\alpha_i^{k-1}-\omega_{ij}\alpha_i^k\alpha_j^{k-1},\ 1\leq i,j\leq n, 1\leq k\leq n-1\right\rangle.
\end{equation}

For example, $B^4_{\omega}$ has a presentation as the path algebras of
\begin{equation}\nonumber
\xymatrix{{\bullet}^1\ar@<-4.5ex>[rr]^{\alpha_4^1}\ar@<-1.5ex>[rr]^{\alpha_{3}^1}\ar@<1.5ex>[rr]^{\alpha_{2}^1}\ar@<4.5ex>[rr]^{\alpha_{1}^1}&&{\bullet}^2\ar@<-4.5ex>[rr]^{\alpha_4^2}\ar@<-1.5ex>[rr]^{\alpha_{3}^2}\ar@<1.5ex>[rr]^{\alpha_{2}^2}\ar@<4.5ex>[rr]^{\alpha_{1}^2}&&{\bullet}^3\ar@<-4.5ex>[rr]^{\alpha_4^3}\ar@<-1.5ex>[rr]^{\alpha_{3}^3}\ar@<1.5ex>[rr]^{\alpha_{2}^3}\ar@<4.5ex>[rr]^{\alpha_{1}^3}&&{\bullet}^4}
\end{equation}
with relations $\alpha^k_j\alpha^{k-1}_i=\omega_{ij}\alpha^k_i\alpha^{k-1}_j$, where $k\in\left\{2,3\right\}$ and $1\leq i\neq j \leq 4$.

It is easy to see that such a strong exceptional sequence forms a basis for the Grothendieck group of the derived category (see, for example, \cite{BrSt}). Therefore, the number of elements in such sequence is preserved via derived equivalence.
\begin{lemma}
If $D^b(tails(S^n_{\omega}))\cong D^b(tails(S^m_{\omega'}))$ then $m=n$.
\end{lemma}

In particular if the categories themselves are equivalent (and therefore derived equivalent) we get $m=n$. From now on we are thus interested in comparing the categories associated to $S^n_{\omega}$ and $S^n_{\omega'}$. The superscript $n$ will be dropped and assumed to be fixed.

\end{section}
\begin{section}{Graded equivalences}
Our target is to classify up to equivalence the categories $Gr(S^n_{\omega})$. For this we will use Zhang's theoy of twisting systems (\cite{Zh}) which we quickly review now. Let $R$ and $S$ be connected $\mathbb{N}_0$-graded right Noetherian $\mathbb{K}$-algebras. 

\begin{definition}[Zhang]
A twisting system is a set $\tau=\left\{\tau_n: n\in\mathbb{N}_0\right\}$ of $\mathbb{K}$-linear, degree preserving bijections from $R$ to $R$ satisfying:
\begin{equation}\nonumber
\tau_n(y\tau_m(z))=\tau_n(y)\tau_{n+m}(z),\ \forall l, m, n\in\mathbb{N}_0, y\in R_m,z\in R_l.
\end{equation} 

Given a twisting system over $R$ we can construct a new algebra structure on the underlying vector space $R$, the twisted algebra, by defining new multiplication $\bullet$:
\begin{equation}\nonumber
y\bullet z := y\tau_m(z),\  \forall y\in R_m, z\in R_l.
\end{equation}
\end{definition}

These two algebra structures are closely related as follows.
\begin{theorem}[Zhang, \cite{Zh}]\label{Zhang twist}
For $R$ and $S$ as above and such that $R_1\neq 0$, the following are equivalent:

\begin{enumerate}
\item $R$ is isomorphic (as a graded algebra) to a twist of $S$;
\item $Gr(R)$ is equivalent to $Gr(S)$;
\item There is an equivalence $\Phi$ between $Tails(R)$ and $Tails(S)$ such that shifts of the structure sheaf are preserved, i.e., $\Phi(\pi R(n))=\pi S(n)$, for all  $n\in\mathbb{Z}$.
\end{enumerate}
\end{theorem}

The first natural step towards our target is, however, to classify these algebras up to isomorphism.

\begin{lemma}\label{isomorphism}
$S_{\omega}$ is isomorphic as a graded algebra to $S_{\omega'}$ if and only if there is $\sigma\in \Sigma_n$ such that $\omega'_{\sigma(i)\sigma(j)}=\omega_{ij}$, for all $i,j\in\left\{1,...,n\right\}$.
\end{lemma}
\begin{proof}

Let $\Psi$ be a graded isomorphism from $S_{\omega}$ to $S_{\omega'}$ and let $C$ be the corresponding element in $GL_n(\mathbb{C})$ such that $\Psi(X_j)=\sum_i c_{ij}X_i$. The matrix $C$ induces a graded endomorphism $\Phi$ of the free algebra $F$ in $n$ variables, thus making the following diagram commute:
\begin{equation}\nonumber
\xymatrix{F\ar[d]_{p_{\omega}}\ar[r]^{\Phi}&F\ar[d]^{p_{\omega'}} \\ F/I_{\omega}=S_{\omega}\ar[r]_{\Psi}& S_{\omega'}=F/I_{\omega'}}.
\end{equation}

Thus, we have $\Phi(I_{\omega})=I_{\omega'}$. Given that $\Phi$ is graded, the images of the standard generators of $I_{\omega}$ are of degree two and thus they are linear combinations of the standard generators of $I_{\omega'}$. In particular the coefficients of the image of such a relation at $X_k^2$ have to be zero, for all $k$. From similar simple observations, we get the following equations:
\begin{equation}\label{equation1}
c_{ki}=0 \vee c_{kj}=0 \vee \omega_{ij}=1,\ \forall 1\leq i<j\leq n,\ \forall 1\leq k \leq n
\end{equation}
\begin{equation}\label{equation2}
c_{kj}c_{li}-\omega_{ij}c_{ki}c_{lj}=\omega'_{kl}(\omega_{ij}c_{li}c_{kj}-c_{lj}c_{ki}).
\end{equation}

Suppose that $\omega_{ij}\neq 1$, for all $1\leq i<j\leq n$. Then for any fixed row, equations (\ref{equation1}) guarantee that given any two entries, one of them is necessarily zero. Thus in each row there is a unique nonzero entry and, since the matrix is invertible, the same applies to columns. Thus, there is $\sigma\in \Sigma_n$ such that $\Phi(X_i)=c_{\sigma(i)i}X_{\sigma(i)}$. So we get:
\begin{equation}\nonumber
\Phi(X_jX_i - \omega_{ij}X_iX_j)=c_{\sigma(i)i}c_{\sigma(j)j}(X_{\sigma(j)}X_{\sigma(i)}-\omega_{ij}X_{\sigma(i)}X_{\sigma(j)})
\end{equation}
and therefore $\omega'_{\sigma(i)\sigma(j)}=\omega_{ij}$, since $\Phi$ is linear.

If some of the parameters $\omega_{ij}$ are 1, the matrix $C$ might not be of the form above described. Note however that, for $\omega_{ij}=1$, equation (\ref{equation2}) is the same as $M_{k,l,i,j}=\omega'_{kl}M_{k,l,i,j}$, where $M_{k,l,i,j}$ stands for the $2 \times 2$ minor given by rows $k,l$ and columns $i,j$. As $C$ is invertible, for each $i,j$ such that $\omega_{ij}=1$, there is $k,l$ such that $M_{k,l,i,j}\neq 0$. Moreover, the nonsingularity of $C$ assures that we can make these choices so that they do not coincide for distinct pairs $i,j$ with $\omega_{ij}=1$.

This means that we can create a new matrix $C'$ by deleting some entries of $C$ (i.e., setting them to be zero), leaving the chosen nonzero $2\times 2$ minors, in such a way that $C'\in GL_n(\mathbb{C})$ and satisfies:
\begin{equation}\nonumber
c'_{ki}=0 \vee c'_{kj}=0,\ \forall 1\leq i<j\leq n,\ \forall 1\leq k \leq n.
\end{equation}
Hence, equations (\ref{equation1}) and (\ref{equation2}) hold for $C'$ and thus $C'$ provides an isomorphism $\tilde{\Psi}$ from $S_{\omega}$ to $S_{\omega'}$ of the form $\tilde{\Psi}(X_i)=c'_{\sigma(i)i}X_{\sigma(i)}$, as above. This proves that $\omega'_{\sigma(i)\sigma(j)}=\omega_{ij}$. The converse follows from the construction of $\tilde{\Psi}$.

\end{proof}

The following definition will prove useful and, later, natural. 

\begin{definition}
Given $\tau\in \Sigma_n$ a k-cycle, $k\leq n$, define the $\tau$-cyclic $q$-number by $q_{\tau}(\omega):= \prod_{i=1}^k \omega_{\tau^{i-1}(v)\tau^i(v)}$, for any $1\leq v\leq n$ which is not fixed by $\tau$.
\end{definition}

It is clear that the definition does not depend on the choice of $v$.

\begin{theorem}\label{graded equivalence}
The following conditions are equivalent:

\begin{enumerate}
\item $\exists\sigma\in \Sigma_n,\ \exists (m_1,...,m_n)\in\mathbb{C}^{*n}: \ \omega'_{\sigma(i)\sigma(j)}= m_im_j^{-1} \omega_{ij}$;
\item $Gr(S_{\omega}) \cong Gr(S_{\omega'})$;
\item There is an equivalence between $Tails(S_{\omega})$ and $Tails(S_{\omega'})$ preserving the shifts of the structure sheaf;
\item $B_{\omega} \cong B_{\omega'}$;
\item $\exists\sigma\in \Sigma_n: \forall \tau\in \Sigma_n$, $\tau$ k-cycle: $q_{\tau}(\omega')=q_{\sigma^{-1}\tau\sigma}(\omega)$;
\item $Mod(B_{\omega})\cong Mod(B_{\omega'})$.

\end{enumerate}
\end{theorem}
\begin{remark}\label{length3}
Note that it is sufficient to check condition (5) for cycles of length 3. In fact, if $k<n$, $\tau=(a_1a_2...a_k)$, and $b$ is fixed by $\tau$, then 
\begin{equation}\nonumber
q_{\tau}(\omega)=(\prod_{i=1}^k q_{(ba_ia_{i+1})})q_{(ba_ka_1)}.
\end{equation}
If $\tau$ is of maximal length (i.e., $k=n$) we can write
\begin{equation}\nonumber
q_{\tau}(\omega)=q_{(a_1...a_{n-1})}(\omega)q_{(a_1a_{n-1}a_{n})}(\omega)
\end{equation} 
and then repeat the previous step to write $q_{(a_1...a_{n-1})}(\omega)$ (and thus $q_{\tau}(\omega)$) as a product of $q$-numbers of length 3. This argument shows that we only need the information provided by the cyclic $q$-numbers of length 3, $q_{(abc)}(\omega)$, with fixed $a$.
\end{remark}

Let us now prove the proposition.

\begin{proof}
Theorem \ref{Zhang twist} tells us that (2)$\Leftrightarrow$(3). We shall prove (1)$\Leftrightarrow$(2), (1)$\Leftrightarrow$(5), (1)$\Rightarrow$(4), (4)$\Rightarrow$(3) and (4)$\Leftrightarrow$(6).

Let us start with (1)$\Rightarrow$ (2). Suppose that (1) holds. Let $f$ be the algebra automorphism of $S_{\omega}$ defined by $f(X_i)=m_iX_i$, for all $1\leq i \leq n$. Consider the twisting system $\left\{f^n: n\in\mathbb{Z}\right\}$. The twisted algebra is such that $X_i\bullet X_j = m_jX_iX_j$ and $X_j\bullet X_i=m_iX_jX_i$. Since $X_jX_i = \omega_{ij}X_iX_j$ we get:
\begin{equation}\nonumber
X_j\bullet X_i = m_i\omega_{ij}X_iX_j=m_im_j^{-1}\omega_{ij}X_i\bullet X_j = \tilde{\omega}_{ij}X_i\bullet X_j
\end{equation}
and thus the twisted algebra is just $S_{\tilde{\omega}}$ which, by the previous lemma, is isomorphic to $S_{\omega'}$. By Zhang's theorem \ref{Zhang twist} we have (2).

To prove the converse, (2)$\Rightarrow$ (1), we start by observing that, since $S^n_{\omega}$ is generated in degree 1, for a twisting system $\tau$ we have
\begin{equation}\label{gen twist}
(S^n_{\omega})^{\tau}=\mathbb{C}\left\langle X_1,...,X_n\right\rangle/\left\langle X_j\tau_1^{-1}(X_i)-\omega_{ij}X_i\tau_1^{-1}(X_j),\ 1\leq i<j \leq n\right\rangle
\end{equation}
and thus the twisted algebra is completely determined by how $\tau_1$ acts in the degree 1 component of $S_{\omega}^n$, $(S_{\omega}^n)_1$.  Let $\tau_1$ act in $(S_{\omega}^n)_1$ by a matrix $C\in GL_n(\mathbb{C})$. We observe that the twisted algebra induced by any conjugate of $C$ in $GL_n(\mathbb{C})$ is isomorphic to the twisted algebra induced by $C$. Indeed, given $P=(p_{ij})_{1\leq i,j\leq n} \in GL_n(\mathbb{C})$ let $Y_1,...,Y_n$ be the basis of $(S_{\omega}^n)_1$ such that $Y_i=\sum_{j=1}^np_{ij}X_j$ and let $A$ be an algebra such that $P$ induces an algebra isomorphism between $S_{\omega}^n$ and $A$ ($A$ can be presented as the algebra generated by $Y_1,...,Y_n$ with relations obtained by rewriting the each $X_i$ in the standard relations of $S_{\omega}^n$ in the new basis). Now let $\tilde{\tau}_1$ be such that its action in $A_1$ is determined by $PCP^{-1}$. Thus, as linear maps on $A_1\cong (S_{\omega})_1$, $\tau_1$ and $\tilde{\tau}_1$ coincide. Hence, if $\tilde{\tau}$ is a twisting system of $A$ containing $\tilde{\tau}_1$ we have $A^{\tilde{\tau}}\cong S_{\omega}^{\tau}$. But since $S_{\omega}^n$ and $A$ are isomorphic as graded algebras, we also have $S_{\omega}^{\tau}\cong S_{\omega}^{\tilde{\tau}}$. Hence, we may assume without loss of generality that $C$ is in Jordan normal form, i.e., 

\begin{equation}\label{Jordan}
C=\left(\begin{array}{ccccc}m_1 & 0 & 0 &... & 0 \\ \delta_1 & m_2 & 0 & ... & 0\\ 0 & \delta_2 & m_3 & ... & 0 \\ ...& ... & ...& ... &... \\ 0& ... & 0 & \delta_{n-1} & m_n \end{array}\right),
\end{equation}
where $m_1,...,m_n \in \mathbb{C}^*$ and $\delta_1,...,\delta_{n-1}\in\left\{0,1\right\}$.
Note also that, for a matrix of this form, $\tau_1^{-1}(X_i)$ is a linear combination of $X_1,...,X_{i-1}$.

We say that the set of parameters $\omega$ is non-repetitive if $\omega_{ij} \neq \omega_{kl}$ for any $1\leq i < j \leq n$ and $1\leq k < l \leq n$. It is easy to see that for a non-repetitive set of parameters, the only normal elements (i.e., elements $x$ of $S_{\omega}$ such that $xS_{\omega}=S_{\omega}x$) of degree 1 are scalar multiples of $X_1,X_2,...,X_n$. 
For that purpose we recall the following result from \cite{QFD}.
\begin{lemma}[Launois, Lenagan, Rigal,\cite{QFD}]
Let $R$ be a prime noetherian ring and suppose that $d$, $s$ are normal elements of $R$ such that $dR$ is prime and $s\notin dR$. Then, there is a unit $v\in R$ such that $sd = vds.$
\end{lemma}
It is clear that $X_k$ generates a (completely) prime ideal since the respective factor algebra is a domain (it is of the form $S^{n-1}_{\omega|}$ for some $\omega|$ a restriction of parameters). Since all invertible elements are scalars, if $\sum_{i=1}^n a_iX_i$ is normal then, by the lemma above,
\begin{equation}\nonumber
X_k\sum_{i=1}^n a_iX_i = (\sum_{i=0}^na_i\omega_{ik}X_i)X_k = \lambda_k(\sum_{i=1}^n a_iX_i)X_k
\end{equation} 
which implies that for all $a_i\neq 0$, $\omega_{ik}=\lambda_k$. Since $\omega$ is non-repetitive, all but one of the coefficients $a_i$ are equal to zero.

Note that for every $S_\omega$ there is some non-repetitive set of parameters $\tilde{\omega}$ such that $Gr(S_{\omega})\cong Gr(S_{\tilde{\omega}})$. In fact we can obtain such an algebra $S_{\tilde{\omega}}$ by choosing a suitable matrix $C$ of the form (\ref{Jordan}) with $\delta_1=...=\delta_{n-1}=0$  (i.e., a diagonal matrix) determining a twisting system. To do this, consider $G$ to be the finitely generated subgroup of $\mathbb{C}^*$ generated by $\omega_{ij}$, $1\leq i,j\leq n$. Clearly one can choose $n$ elements of $\mathbb{C}^*$, $a_1,...,a_n$, such that they generate a subgroup $H$ of $\mathbb{C}^*$ of rank $n$ intersecting $G$ trivially (if this was not the case, $G$ would necessarily be infinitely generated). Define $C$ as the diagonal matrix formed by these numbers. The twisted algebra is such that (see $(1)\Rightarrow (2)$) $\tilde{\omega}_{ij}=a_ia_j^{-1}\omega_{ij}$. If $\tilde{\omega}_{ij}=\tilde{\omega}_{kl}$ then $a_ia_j^{-1}a_ka_l^{-1}=\omega_{ij}^{-1}\omega{kl}\in G$ and since $H$ has rank $n$, $k=j$ and $l=i$ and therefore $\tilde{\omega}$ is non-repetitive.

Now if $Gr(S_{\omega}) \cong Gr(S_{\omega'})$, then $Gr(S_{\omega}) \cong Gr(S_{\tilde{\omega}})$ for some non-repetitive set of parameters $\tilde{\omega}$ defined by $\tilde{\omega}_{ij}=a_ia_j^{-1}\omega'_{ij}$ for some $a_1,...,a_n\in\mathbb{C}^*$. Therefore, if we prove that there is $\sigma\in\Sigma_n$ such that $\tilde{\omega}_{\sigma(i)\sigma(j)}=b_ib_j^{-1}\omega_{ij}$ then we also have
\begin{equation}\nonumber
\omega'_{\sigma(i)\sigma(j)}=a_{\sigma(j)}a_{\sigma(i)}^{-1}\tilde{\omega}_{\sigma(i)\sigma(j)}=a_{\sigma(j)}a_{\sigma(i)}^{-1}b_{i}b_j^{-1}\omega_{ij}
\end{equation}
which implies that, by taking $m_i=a_{\sigma(i)}^{-1}b_{i}$, we get the wanted result. It is then enough to prove the result for $S_{\omega'}$ with non-repetitive $\omega'$.

Suppose $Gr(S_{\omega}) \cong Gr(S_{\omega'})$ with non-repetitive $\omega'$. By theorem \ref{Zhang twist} $S_{\omega'}\cong (S_{\omega})^{\tau}$ for some twisting system $\tau$ and, moreover, by our argument above $\tau$ can be taken to be determined by a matrix $C$ in Jordan normal form (\ref{Jordan}). This implies, in particular, that $X_1,...,X_n$ form a polynormal sequence in $(S_{\omega})^{\tau}$, i.e., $X_1$ is normal and $X_i$ is normal in the quotient ring $(S_{\omega})^{\tau}/ \left\langle X_1,...,X_{i-1}\right\rangle$. This is due to the fact that $\tau_1^{-1}(X_i)=m_i^{-1}X_i + f_i$ where $f_i$ is a linear combination of $X_1,...,X_{i-1}$. By the observations above, since $\omega'$ is non-repetitive, $(S_{\omega})^{\tau}$ has, up to scalar multiples, exactly $n$ normal elements of degree 1. Therefore, if $\psi$ is an isomorphism from $(S^n_{\omega})^{\tau}$ to $S_{\omega'}$ there is $1\leq k\leq n$ such $\psi{X_1}=\lambda X_k$ for some $\lambda\in\mathbb{C}^*$  and thus $\psi$ establishes an isomorphism from $(S^n_{\omega})^{\tau}/X_1(S^n_{\omega})^{\tau}$ to $S^{n-1}_{\omega'|}$ where $\omega'|$ is the set of parameters naturally obtained by restriction, i.e., by forgetting $\omega_{ik}$ and $\omega_{kj}$ for all $i<k\leq n$ and $n\geq j>k$ (and thus $\omega'|$ is non-repetitive). Similarly, each factor ring $(S_{\omega})^{\tau}/ \left\langle X_1,...,X_{i-1}\right\rangle$ is isomorphic to $S^{n-i+1}_{\omega'|}$ for some suitable restriction of $\omega'$ and therefore it has, up to scalar multiples, precisely $n-i+1$ normal elements of degree 1. Since the projections of the normal elements of degree 1 of $((S_{\omega})^{\tau})_1$ in such a factor ring are normal as well (and there are $n-i+1$ of them), we conclude that these projections coincide, up to scalar multiples, with the normal elements of degree 1 in the factor ring. Therefore, in $((S_{\omega})^{\tau})_1$ the normal elements must be of the form:
\begin{equation}\label{basis}
\begin{array}{c}Y_1=X_1\\ Y_2=X_2+\lambda_{21}X_1\\ Y_3=X_3 + \lambda_{32}X_2 + \lambda_{31}X_1\\ ...\\ Y_n=X_n + \sum_{k=1}^{n-1}\lambda_{nk}X_k\end{array}.
\end{equation}
Clearly we must have $Y_jY_i=\omega'_{\sigma(i)\sigma(j)}Y_iY_j$ for some $\sigma\in\Sigma_n$ by lemma \ref{isomorphism}. 

We are now ready to prove that, if $(S_{\omega})^{\tau}\cong S_{\omega'}$ with $\tau_1$ acting in $(S_{\omega})_1$ by a matrix $C$ of the form $\ref{Jordan}$ and with $\omega'$ non-repetitive then there is $\sigma\in\Sigma_n$ such that $\omega'_{\sigma(i)\sigma(j)}=m_im_j^{-1}$, the $m_1,...,m_n$ being the elements in the diagonal of $C$. For this we use induction on the number $n$ of variables. For $n=2$ it is well-known that $S^2_{\omega}$ is a always twist of the polynomial ring in two variables given by the $\tau_1(X_1)=X_1$ and $\tau_1(X_2)=\omega_{12}^{-1}X_2$. More generally, one can twist $S^2_{\omega}$ to $S^2_{\omega'}$ by defining $\tau_1(X_1)=X_1$ and $\tau_1(X_2)= \omega'_{12}\omega_{12}^{-1}X_2$. Moreover, if the twisted algebra is determined by a non-diagonalisable matrix, an easy calculation shows that the parameter of the twisted algebra remains the same (i.e., $\omega'_{12}=mm^{-1}\omega_{12}$ where $m$ is the element appearing twice in the diagonal of the Jordan normal form). Thus the result holds for $n=2$. Suppose now the result is valid for $n-1$ and let $\psi$ be an isomorphism from $(S^n_{\omega})^{\tau}$ to $S^n_{\omega'}$ with $\tau$ and $\omega'$ are as before. Then, applying our induction hypothesis to the quotient $(S^n_{\omega})^{\tau}/\left\langle X_1\right\rangle\cong S^{n-1}_{\omega'|}$, where $\omega'|$ is a suitable restriction of $\omega'$ (forgetting $X_k=\psi(X_1)$), we conclude that there is a bijection $\tilde{\sigma}$ from $\left\{2,...,n\right\}$ to $\left\{1,...,\hat{k},...,n\right\}$ such that $\omega'_{\tilde{\sigma}(i)\tilde{\sigma}(j)}=m_im_j^{-1}\omega_{ij}$, where $m_2,...,m_n$ are the elements in the diagonal of (\ref{Jordan}). Now, consider the basis of normal elements of $(S^n_{\omega})^{\tau}$ given by  $Y_1,...,Y_n$ as in (\ref{basis}). We have
\begin{equation}\nonumber
Y_1Y_j=\omega'_{\sigma(j)\sigma(1)} Y_jY_1
\end{equation}
where $\sigma$ is determined by $\tilde{\sigma}$ such that $\sigma(1)=k$. Expanding this equation we get
\begin{equation}\nonumber
X_1(X_j + \sum_{k=1}^{j-1}\lambda_{jk}X_k)=\omega'_{\sigma(j)\sigma(1)} (X_j + \sum_{k=1}^{j-1}\lambda_{jk}X_k)X_1
\end{equation}
which by using the relations in $(S^n_{\omega})^{\tau}$ means that 
\begin{equation}\nonumber
X_1(X_j + \sum_{j=1}^{j-1}\lambda_{jk}X_k)=\omega'_{\sigma(j)\sigma(1)} m_1X_1(\omega_{1j}m_j^{-1}X_j + f_j +\sum_{k=1}^{j-1} \omega_{1k}m_k^{-1}\lambda_{jk}X_k + f_k)
\end{equation}
where each $f_k$ is a linear combination of $X_1,...,X_{k-1}$. Therefore, by linear independence of the terms, looking at the coefficients of $X_1X_n$ we conclude that 
\begin{equation}\nonumber
1=\omega'_{\sigma(j)\sigma(1)}m_1m_j^{-1}\omega_{ij}
\end{equation}
which finishes our proof.

Auroux, Katzarkov and Orlov proved (4)$\Rightarrow$(3) and we recall the main ingredients of their argument (\cite{AKO}). Suppose $B_{\omega}\cong B_{\omega'}$ via $\Phi$ and   
 consider the chain of equivalences given by
\begin{equation}\nonumber
D^b(tails(S_{\omega}))\longrightarrow D^b(mod(B_{\omega}))\longrightarrow D^b(mod(B_{\omega'}))\longrightarrow D^b(tails(S_{\omega'})).
\end{equation}
They prove that such chain takes $O_{\omega}(i)$ to $O_{\omega'}(i)$, for all $i\in \mathbb{Z}$ (this fact is clear for $0\leq i\leq n-1$ since the middle equivalence is induced from $\Phi$). Observing that this sequence is ample (\cite{AKO}), and that it is preserved by the functor, a previous result by Bondal and Orlov (\cite{BO2}, \cite{Or2}) regarding autoequivalences preserving an ample sequence yields the result. A good account of this result can also be found in Huybrechts' book (\cite{Huy}).

We will now prove (1)$\Rightarrow$(4). Let $\sigma\in \Sigma_n$ and $m_1, ..., m_n$ as in (1). Choose $a_i^1$ and $a_i^2$ such that $m_i = \frac{a_i^2}{a_i^1}$ and inductively define $a_i^k=\frac{a_i^{k-2}}{a_i^{k-1}}$. Further define $\Phi: B_{\omega}\longrightarrow B_{\omega'}$ by $\Phi(\alpha^k_l)=a_l^k\alpha^k_{\sigma(l)}$ and (1) easily implies that the ideal of relations in $B_{\omega}$ is mapped to the ideal of relations in $B_{\omega'}$, thus making $\Phi$ an isomorphism. 

It is straightforward to check that (1)$\Rightarrow$(5): in fact, following remark \ref{length3} take $\sigma$ as in (1) and, given a 3-cycle $\tau=(abc)$,
\begin{equation}\nonumber
q_{\tau}(\omega')=\omega'_{ab}\omega'_{bc}\omega'_{ca} 
\end{equation}
\begin{equation}\nonumber
=\frac{m_a}{m_{b}}\omega_{\sigma^{-1}(a)\sigma^{-1}(b)}\frac{m_{b}}{m_{c}}\omega_{\sigma^{-1}(b)\sigma^{-1}(c)}\frac{m_{c}}{m_a}\omega_{\sigma^{-1}(c)\sigma^{-1}(a)}
\end{equation}
\begin{equation}\nonumber
=\omega_{\sigma^{-1}(a)\sigma^{-1}(b)}\omega_{\sigma^{-1}(b)\sigma^{-1}(c)}\omega_{\sigma^{-1}(c)\sigma^{-1}(a)} = q_{\sigma^{-1}\tau\sigma}(\omega).
\end{equation}

Conversely, note first that the existence of a collection of $\left\{m_i\right\}_{1\leq i\leq n}$ as in (1) is equivalent to the existence of a collection $\left\{\lambda_{ij}\right\}_{1\leq i,j\leq n}$ of nonzero complex numbers such that
\begin{equation}\label{equation3}
\lambda_{ij}\lambda_{jk}=\lambda_{ik}\ \ \ {\rm and}\ \ \ \lambda_{ij}\lambda_{ji}=1
\end{equation}
where $\lambda_{ij}=m_im_j^{-1}$ (and given such collection, we can choose $m_1$ and set $m_j=m_1\lambda_{1j}^{-1}=m_1\lambda_{j1}$). Let us assume (5). To prove (1), we shall fix $\sigma$ as in (5) and find appropriate $\lambda_{ij}$'s satisfying (\ref{equation3}) such that $\omega'_{\sigma(i)\sigma(j)}= \lambda_{ij} \omega_{ij}$. By remark \ref{length3} it is enough to consider 3-cycles (which simplifies notation a great deal).

Fix $1<k<l<s\leq n$ and let $\sigma(i)=k$, $\sigma(j)=l$ and $\sigma(t)=s$. Consider the cycle of length $3$ $(1kl)$. By hypothesis we have $q_{(1kl)}(\omega')=q_{\sigma^{-1}(1kl)\sigma}(\omega)$ and thus we can write
\begin{equation}
\omega'_{kl}=\underbrace{\frac{\omega_{\sigma^{-1}(1)i}\omega_{j\sigma^{-1}(1)}}{\omega'_{1k}\omega'_{l1}}}_{=:\lambda_{ij}}\ \omega_{ij}.
\end{equation}
Similarly, by considering the cycle $(1ls)$, one gets 
\begin{equation}
\lambda_{jt}= \frac{\omega_{\sigma^{-1}(1)j}\omega_{t\sigma^{-1}(1)}}{\omega'_{1l}\omega'_{s1}}
\end{equation}
and, by considering $(1ks)$,
\begin{equation}
\lambda_{it}= \frac{\omega_{\sigma^{-1}(1)i}\omega_{t\sigma^{-1}(1)}}{\omega'_{1k}\omega'_{s1}}.
\end{equation}

We need to prove that the definition of $\lambda_{ij}$ does not depend on the choice of the cycle and that they satisfy equation (\ref{equation3}). This is the same as showing that
\begin{equation}\nonumber
\frac{\omega_{\sigma^{-1}(1)i}\omega_{j\sigma^{-1}(1)}}{\omega'_{1k}\omega'_{l1}}=\frac{\omega_{\sigma^{-1}(v)i}\omega_{j\sigma^{-1}(v)}}{\omega'_{vk}\omega'_{lv}}
\end{equation}
for any $v\notin \left\{1,k,l\right\}$. This is clear since the equation above is equivalent to
\begin{equation}\nonumber
\frac{q_{(vkl)}(\omega')/\omega'_{kl}}{q_{(1kl)}(\omega')/\omega'_{kl}}=\frac{q_{(\sigma^{-1}(v)ij)}(\omega)/\omega_{ij}}{q_{(\sigma^{-1}(1)ij)}(\omega)/\omega_{ij}}
\end{equation}
given that, by hypothesis, $q_{(vkl)}(\omega')=q_{(\sigma^{-1}(v)ij)}(\omega)$ and $q_{(1kl)}(\omega')=q_{(\sigma^{-1}(1)ij)}(\omega)$. 

Finally we see that
\begin{equation}\nonumber
\lambda_{ij}\lambda_{jt}= \frac{\omega_{\sigma^{-1}(1)i}\omega_{j\sigma^{-1}(1)}}{\omega'_{1k}\omega'_{l1}}\frac{\omega_{\sigma^{-1}(1)j}\omega_{t\sigma^{-1}(1)}}{\omega'_{1l}\omega'_{s1}} = \frac{\omega_{\sigma^{-1}(1)i}\omega_{t\sigma^{-1}(1)}}{\omega'_{1k}\omega'_{s1}} =\lambda_{it}
\end{equation}
since $\omega_{ab}\omega_{ba}=1$, for all $a$ and $b$. Thus we get (1).

Clearly (4) implies (6). The converse follows from the fact that $B_{\omega}$ is a basic algebra. Indeed, a progenerator in $mod(B_{\omega})$ such that its endomorphism algebra is $B_{\omega'}$ (thus basic as well) has to be $B_{\omega}$. Thus Morita equivalence in this case implies isomorphism (see \cite{Erd} for more details).
\end{proof}
\begin{remark}\nonumber
Minamoto and Mori have recently showed that the graded Morita equivalence classes within certain families of algebras ($S^n_{\omega}$ in our case) depend only on the isomorphism classes of certain related finite dimensional algebras ($B^n_{\omega}$ in our case). In that sense, their results generalise the previous proposition (\cite{MM}).
\end{remark}

\end{section}
\begin{section}{Birational equivalence}

The context in which these q-cyclic numbers, defined in the previous section, appear naturally is explored below. They actually concern the birational classification of these spaces.  

\begin{definition}
The function division ring of $\mathbb{P}^{n-1}_{\omega}$ is defined to be
\begin{equation}\nonumber
\mathbb{C}(\mathbb{P}^{n-1}_{\omega}):=Frac_{Gr}(S^n_{\omega})_0=\left\{fg^{-1}: f, g\in h(S^n_{\omega}), g\neq 0, deg(f) = deg(g)\right\}.
\end{equation} 
$\mathbb{P}^{n-1}_{\omega}$ and $\mathbb{P}^{n-1}_{\omega'}$ are said to be birationally equivalent if $\mathbb{C}(\mathbb{P}^{n-1}_{\omega})\cong \mathbb{C}(\mathbb{P}^{n-1}_{\omega'})$.
\end{definition}

Let $T_{\omega}$ be the algebra of quantum Laurent polynomials containing $S_{\omega}$ as a subalgebra, i.e., $T_{\omega}=\mathbb{C}\left\langle X_1^{\pm 1},...,X_n^{\pm 1}\right\rangle/\left\langle X_jX_i-\omega_{ij}X_iX_j,\ i,j\in\left\{1,...,n\right\}\right\rangle$. Richard classified, up to isomorphism, this family of algebras, which are called quantum tori (\cite{Rich}).

\begin{lemma}[Richard, \cite{Rich}]\label{Richard}
\begin{enumerate}
\item $T_{\omega}\cong T_{\omega'}$ if and only if there is a matrix $A=(a_{ij})\in GL_n(\mathbb{Z})$ such that $\omega'_{ij}=\prod\limits_{1\leq k, l\leq n} \omega_{kt}^{a_{ki}a_{tj}}$;
\item $T_{\omega}$ is simple if and only if there is not a nonzero vector $a=(a_1,...,a_n)\in \mathbb{Z}^n$ such that, for all $1\leq j\leq n$, $\omega_{1j}^{a_1}...\omega_{nj}^{a_n}=1$;
\item If $T_{\omega}$ is simple, then $T_{\omega}\cong T_{\omega'}$ if and only if $Frac(T_{\omega})\cong Frac(T_{\omega'})$, where $Frac(T_{\omega})$ is the division ring of right fractions of $T_{\omega}$.
\end{enumerate}
\end{lemma}

We shall refer to $\omega$ as \textit{generic} whenever $T_{\omega}$ is simple.

\begin{remark}\label{q-com is frac-com}
Observe an important fact about $q$-cyclic numbers: they show up when commuting fractions as follows:
\begin{equation}\label{frac qcom}
X_jX_k^{-1}X_iX_k^{-1}=\omega_{ki}X_jX_iX_k^{-2}=\omega_{ki}\omega_{ij}X_iX_jX_k^{-2}=
\end{equation}
\begin{equation}\nonumber
=\omega_{ki}\omega_{ij}\omega_{jk}X_iX_{k}^{-1}X_jX_k^{-1}=q_{(ijk)}(\omega)X_iX_{k}^{-1}X_jX_k^{-1}.
\end{equation}
\end{remark}

Define $^kS_{q}:=(S_{\omega}[X_k^{-1}])_0$. By the above, It is easy to see that 
\begin{equation}\nonumber
^kS_{q}\cong\mathbb{C}\left\langle Y_1,...,\hat{Y_k},...,Y_{n}\right\rangle/\left\langle Y_jY_i-q_{(kij)}(\omega)Y_iY_j,\ \forall 1\leq i,j\leq n, i,j\neq k\right\rangle
\end{equation}
where $Y_i=X_iX_k^{-1}$. Let $T_q$ be the torus associated with $^kS_q$, i.e., the algebra of Laurent polynomials containing $^kS_q$ as a subalgebra.

\begin{remark}
As the notation suggests, the isomorphism class of $T_q$ does not depend on $k$. This can be shown using lemma \ref{Richard} and keeping in mind relations of the type $q_{(234)}=q_{(123)}q_{(134)}q_{(142)}$. In practice, we will take $k=1$.
\end{remark}

In the following proposition we denote $q_{(1ij)}(\omega)$ by $q_{ij}$ and $q_{(1ij)}(\omega')$ by $q'_{ij}$.

\begin{theorem}\label{bir class}
If there is a matrix $A=(a_{ij})_{2\leq i,j\leq n}$ in $GL_{n-1}(\mathbb{Z})$ such that $q_{ij}'= \prod\limits_{2\leq k, l\leq n} q_{kl}^{a_{ki}a_{lj}}$, then $\mathbb{P}^{n-1}_{\omega}$ and $\mathbb{P}^{n-1}_{\omega'}$ are birationally equivalent. Moreover, if $q=(q_{ij})_{1\leq i,j\leq n-1}$ is generic then the converse is also true.
\end{theorem}
\begin{proof}
We shall prove that $Frac(T_q)\cong \mathbb{C}(\mathbb{P}^{n-1}_{\omega})$, where $T_q$ is as previously defined. Consider the map $\phi: Frac(T_q)\longrightarrow \mathbb{C}(\mathbb{P}^{n-1}_{\omega})$ such that $\phi(Y_i)=X_iX_1^{-1}$. Note that this is well defined, as $X_jX_1^{-1}X_iX_1^{-1}= q_{ij}X_iX_1^{-1}X_jX_1^{-1}$ (see equation (\ref{frac qcom})). Clearly $\phi$ is injective as it is a nonzero map defined on a division ring, $Frac(T_q)$. Now, suppose $fg^{-1}\in \mathbb{C}(\mathbb{P}^{n-1}_{\omega})$, with $f$ monomial. Then we can write
\begin{equation}\nonumber
fg^{-1}=(gf^{-1})^{-1}=(\sum_i g_if^{-1})^{-1}
\end{equation}
where the $g_i$'s are monomials as well and deg$(g_i)=$deg$(f)$. Thus $g_if^{-1}$ is a product of powers of some $X_{i_j}X_1^{-1}$'s and therefore $fg^{-1}$ lies in the image of $Frac(T_q)$, proving surjectivity of $\phi$.

If $q$ is generic, $T_q$ is simple. By proposition \ref{Richard}, $\mathbb{C}(\mathbb{P}^{n-1}_{\omega})\cong \mathbb{C}(\mathbb{P}^{n-1}_{\omega'})$ if and only if $T_q\cong T_q'$ (and clearly we only need generic $q$ for the nontrivial implication). The same result also shows that this is the case if and only if there is $A\in GL_{n-1}(\mathbb{Z})$ such that $q_{ij}'= \prod\limits_{1\leq k, l\leq n-1} q_{kl}^{a_{ki}a_{lj}}$, hence finishing the proof. 
\end{proof}

\end{section}

\begin{section}{The point variety}

In this section we compute the point variety (recall definition \ref{pt scheme}) of $\mathbb{P}^n_{\omega}$ in terms of its parameters $\omega_{ij}$. Once again, this depends only on the $q$-cyclic numbers. First we recall a well-known corollary of Artin, Tate and Van den Bergh's work (\cite{ATV}) as observed by Mori (\cite{Mori}).

\begin{lemma}[Artin, Tate, Van den Bergh, \cite{ATV}]\label{ptsquadr}
Let $R=T(V)/I$ be a quadratic graded $\mathbb{C}$-algebra, where $V$ is a finite dimensional complex vector space. If $V(I)\subset \mathbb{P}(V^*)\times \mathbb{P}(V^*)$ is the graph of an automorphism of a closed $\mathbb{C}$-subscheme, $E$, then $E$ is the point scheme of $R$.
\end{lemma}
\begin{proof}
Since $I$ is generated in degree 2, we have $I_3=T_1I_2 + I_2T_1$ and therefore $\Omega_3=\Omega_2\times \mathbb{P}(V^*)\cap \mathbb{P}(V^*)\times \Omega_2$ (\cite{ATV}). Since $\Omega_2=\left\{(x,\sigma(x)): x\in E\right\}$ for some automorphism $\sigma$ of $E$, we get
\begin{equation}\nonumber
\Omega_3=\left\{(x,\sigma(x),\sigma^2(x)): x\in E\right\}.
\end{equation}
Clearly $E\cong\Omega_2\cong\Omega_3$ and, by induction, $E\cong\Omega_d$ for all $d\geq 2$. Thus the point scheme is $E$ as the inverse system is constant.
\end{proof}

We can now compute the point varieties of $S^n_{\omega}$.

\begin{proposition}\label{ptvar}
The point variety of $S^n_{\omega}$ is the subvariety of $\mathbb{P}^{n-1}$ defined by $\bigcap\limits_{q_{(ijk)}(\omega)\neq 1} V(X_iX_jX_k)$, where $V(X_iX_jX_k)$ is the zero locus of $X_iX_jX_k$ in $\mathbb{P}^{n-1}$.
\end{proposition}
\begin{proof}
This proof was first sketched by Hattori (\cite{Hat}). It also uses some ideas from Vancliff (\cite{Van}).
For practical purposes we shall consider the generators of $I^n_{\omega}$ rewritten in the form $\theta_{ji}X_jX_i-\theta_{ij}X_iX_j$ where $\theta$ is a fixed $n\times n$ matrix of parameters in $\mathbb{C}^*$ such that $\theta_{ij}\theta_{ji}^{-1}=\omega_{ij}$.

Following the notation set up in the introduction, we have $\Omega_2\subset \mathbb{P}^{n-1}\times\mathbb{P}^{n-1}$ and let us consider $E_1$ (respectively $E_2$) to be the image of $\Omega_2$ under the projection map on the first (respectively second) component from $\mathbb{P}^{n-1}\times\mathbb{P}^{n-1}$ to $\mathbb{P}^{n-1}$. The first step of this proof will be to determine $E_1$. We have
\begin{equation}\nonumber
E_1=\left\{x\in\mathbb{P}^{n-1}: \exists y\in\mathbb{P}^{n-1}: (x,y)\in\Omega_2\right\}
\end{equation}
and the defining condition can be rewritten in the form $A_x\bar{y}=0$ where $A_x$ is a $\binom{n}{2}\times n$ matrix in the coordinates of $x$ and $\bar{y}$ an $n$-dimensional column vector with entries the coordinates of $y$. Index lines of the matrix $A_x$ by pairs $(i,j)$ with $1\leq i < j \leq n$ and columns by $1\leq k \leq n$. Then we can describe the matrix $A_x=(a_{(i,j)k})$ (where $x=(x_1,...,x_n)$ is ) by setting
\begin{equation}\label{matrixA}
a_{(i,j)i}= \theta_{ji}x_j,\ a_{(i,j)j}=-\theta_{ij}x_i,\ a_{(i,j)k}= 0,\forall k\neq i,j.
\end{equation}

Note that we can recover the generators of $I^n_{\omega}$ in the algebra $S^n_{\omega}$ by taking the entries of the vector $A_XX$, where $X=(X_1,...,X_n)$ is the vector formed by the generators of $S^n_{\omega}$. This is because $I^n_{\omega}$ is generated in degree 2. A similar matrix appears in the original work of Artin and Schelter (\cite{AS}).

For each $x\in E_1$, we need at least a 1-dimensional space of solutions for the equation $A_x\bar{y}=0$ so that we get a point in the projective space. This happens if and only if the rank of $A_x$ is less or equal than $n-1$. The rank of $A_x$ is at least $n-1$ since, given a nonzero $x_i$ we have a diagonal $(n-1)\times(n-1)$ submatrix with entries $\theta_{ij}x_i$, where $1\leq j \leq n$ and $j\neq i$. Therefore $x\in E_1$ if and only if the rank of $A_x$ is exactly $n-1$, i.e., if the $n\times n$ minors of $A_x$ equal to zero.

The focus is then on computing such minors for $A_X$. For this we use Hattori's technique. Note that $A_X$ is such that each row has exactly two nonzero entries (see description (\ref{matrixA})). Let $\tilde{A}_X$ be the a submatrix formed by $n$ rows of $A_X$. We want to compute $det(\tilde{A}_X)$. To $\tilde{A}_X$ we associate a graph $G$ with $n$ vertices, numbered $1$ to $n$, such that there is an edge between two vertices $i$ and $j$ if and only if the line $(i,j)$ is in $\tilde{A}_X$ or, equivalently, if $X_i$ and $X_j$ appear in one row of $\tilde{A}_X$. We denote by $E(G)$ the set of edges of a graph $G$. Each row corresponds to an edge, each column to a vertex and vice-versa. We permute first the rows and then the columns of $\tilde{A}_X$ (but keep the notation) in such a way that, whenever possible, adjacent rows and adjacent columns correspond, respectively to edges and vertices lying in the same connected component of $G$.

If $det(\tilde{A}_X)\neq 0$, then each connected component of $G$ has exactly the same number of edges and vertices. This is equivalent to each connected component containing exactly one cycle. $G$ itself clearly has this property (i.e., there are $n$ rows and $n$ variables). Indeed, if there is a connected component $D$ with less edges than vertices, we can find a rectangular block in $\tilde{A}_X$, and thus $det(\tilde{A}_X)= 0$ since the column vectors forming this block would be linearly dependent.

So we have a square block decomposition of $\tilde{A}_X$ whose blocks are in bijection with the connected components of $G$. We focus on a given block $B$ with connected graph $D$. To compute the determinant of $B$, we regroup the rows and then the columns (but keep the notation) such that the first rows correspond to the edges and the first columns to the vertices on the cycle of $D$ - call it $C$. Then we get a matrix
\begin{equation}\nonumber
B=\left(\begin{array}{ccc}B_1 & 0 \\ * & B_2 \end{array}\right)
\end{equation}
where $B_1$ is the square matrix with the information from $C$ and $B_2$ is the square matrix with the information from the \textit{legs} - denote them by $L$ - of $D$ (i.e., acyclic paths with one vertex but no edges on $C$). Also, by permutation of rows and columns, $B_2$ can be made lower triangular and thus its determinant will be the product of the elements in the diagonal. If we orient the legs such that the source of an edge $e$ (denoted by $s(e)$) is closer to the cycle than the target of $e$ (denoted by $t(e)$), then we can easily compute this product by
\begin{equation}\nonumber
det(B_2)= \prod\limits_{e\in \rm{E(L)}} \theta_{s(e)t(e)}X_{s(e)}.
\end{equation}

The determinant of $B_1$ can be computed by choosing a vertex of the cycle, $v$, and using the Laplace rule along the corresponding column. We orient the cycle clockwise. It is easy to see that we get
\begin{equation}\nonumber
det(B_1)= \prod\limits_{e\in \rm{E(C)}}\theta_{s(e)t(e)}X_{s(e)}-\prod\limits_{e\in \rm{E(C)}}\theta_{t(e)s(e)}X_{t(e)}.
\end{equation}
Note that $det(B)=det(B_1)det(B_2)=0$ if and only if
\begin{equation}\nonumber
\prod\limits_{e\in \rm{E(C)}}\theta_{s(e)t(e)}=\prod\limits_{e\in \rm{E(C)}}\theta_{t(e)s(e)}\ \ \vee\ \ \prod\limits_{e\in E(D)}X_{s(e)}=0
\end{equation}
or equivalently, if $i_1,...,i_l$ are the vertices of the cycle of $C$,
\begin{equation}\nonumber
q_{(i_1...i_l)}(\omega)=1 \vee \prod\limits_{e\in E(D)}X_{s(e)}=0.
\end{equation}

Thus the determinant of $\tilde{A_X}$ is nonzero whenever, for any cycle $(i_1...i_l)$ in the graph $G$, $q_{(i_1...i_l)}(\omega)\neq 1$. Let $\Delta_{\omega}$ be the set of graphs with $n$ vertices and $n$ edges such that each connected component contains exactly one cycle and for each such cycle $(i_1...i_l)$ we have  $q_{(i_1...i_l)}(\omega)\neq 1$. Then, by the above, we get
\begin{equation}\nonumber
E_1=\bigcap\limits_{G\in\Delta_{\omega}}V(\prod\limits_{e\in E(G)}X_{s(e)}).
\end{equation}
We will now show that $(E_1)_{red}$, i.e., the reduced structure of $E_1$, is in fact the same as $\bigcap\limits_{q_{(ijk)}(\omega)\neq 1} V(X_iX_jX_k)$. For this effect we ignore multiplicities in the formula above. Suppose $x\in (E_1)_{red}$ and choose $(ijk)$ such that $q_{(ijk)}(\omega)\neq 1$. Consider $G$ the graph consisting of the triangle with vertices $i, j$ and $k$ and with $n-3$ edges with source on $k$ and target on the remaining $n-3$ vertices. This graph is clearly in $\Delta_{\omega}$ and therefore $x\in V(X_iX_jX_k)$. Conversely, suppose $x\in \bigcap\limits_{q_{(ijk)}(\omega)\neq 1} V(X_iX_jX_k)$. Since any $q$-cyclic number can be written as a product of $q$-cyclic numbers of length 3 (see \ref{length3}), for any graph $G\in\Delta_{\omega}$ we have some triple $(ijk)$ such that $q_{(ijk)}(\omega)\neq 1$ and $i,j,k$ are contained in some cycle of $G$. Hence $x\in (E_1)_{red}$.

It remains to prove that $\Omega = E_1$ and that, therefore, $\Omega_{red}$ can be described as wanted. With a similar argument to the above on the right, we can easily see that $(E_2)_{red} = (E_1)_{red}$. Thus, the right projection following the inverse of the left projection can be regarded as an isomorphism whose graph is $(\Omega_2)_{red}$. 
\begin{lemma}[Le Bruyn, Smith, Van den Bergh, \cite{LSV}]
If $(\Omega_2)_{red}$ is the graph of an isomorphism between $(E_1)_{red}$ and $(E_2)_{red}$ then $\Omega_2$ is the graph of an isomorphism between $E_1$ and $E_2$.
\end{lemma}
This lemma proves that $\Omega_2$ can be regarded as the graph of an automorphism of $E_1$. Then by lemma \ref{ptsquadr} we have that $\Omega=E_1$.

\end{proof}

\end{section}
\begin{section}{Noncommutative projective spaces of dimensions 2 and 3}
In this section we shall study $\mathbb{P}_{\omega}^3$ and compare the results with the ones obtained for $\mathbb{P}_{\omega}^2$ by Mori (\cite{Mori}). We also use proposition \ref{ptvar} to provide examples of noncommutative projective spaces which are not isomorphic (see theorem \ref{invariant points}). It will be useful to determine the possible point varieties of $\mathbb{P}_{\omega}^3$. This is summarised in the following corollary, an easy application of proposition \ref{ptvar}.
\begin{corollary}\label{P3points}
The point variety of $\mathbb{P}_{\omega}^3$ is isomorphic to one of the following:
\begin{enumerate}
\item $\mathbb{P}^3$ if $q_{(123)}=q_{(124)}=q_{(134)}=1$;
\item $V(X_1,X_2)\cup V(X_3)\cup V(X_4)$ if one of the following holds:
\begin{itemize}
\item $q_{(123)}=q_{(124)}=1$ and $q_{(134)}\neq 1$;
\item $q_{(123)}=q_{(134)}=1$ and $q_{(124)}\neq 1$;
\item $q_{(124)}=q_{(134)}=1$ and $q_{(123)}\neq 1$;
\item $q_{(123)}=1$ and $q_{(124)}=q_{(134)}\neq 1$;
\item $q_{(124)}=1$ and $q_{(123)}=q_{(134)}\neq 1$;
\item $q_{(134)}=1$ and $q_{(123)}=q_{(124)}\neq 1$.
\end{itemize}
\item $V(X_{1})\cup V(X_2,X_3)\cup V(X_2,X_4)\cup V(X_3,X_4)$ if one of the following holds
\begin{itemize}
\item $q_{(123)}=1$ and $1\neq q_{(124)}\neq q_{(134)}\neq 1$;
\item $q_{(124)}=1$ and $1\neq q_{(123)}\neq q_{(134)}^{-1} = q_{(143)}\neq 1$;
\item $q_{(134)}=1$ and $1\neq q_{(123)}\neq q_{(124)}\neq 1$;
\item $q_{(123)}, q_{(124)}, q_{(134)} \neq 1$ and $q_{(123)}q_{(134)}=q_{(124)}^{-1}=q_{(142)}$.
\end{itemize}
\item $V(X_1,X_2)\cup V(X_1,X_3)\cup V(X_1,X_4)\cup V(X_2,X_3)\cup V(X_2,X_4)\cup V(X_3,X_4)$ otherwise.
\end{enumerate}
\end{corollary}

Observe that the description above tells us that the point varieties in case 2 ore formed by one line and two hyperplanes, in case 3 by one hyperplane and three lines and in case 4 by six lines.

Concerning $\mathbb{P}_{\omega}^2$ we have the following result. Note that it also solves the classification problem for the family $\mathbb{P}_{\omega}^2$.

\begin{theorem}[Mori, \cite{Mori}]\label{Mori}
The following are equivalent:
\begin{enumerate}
\item $Gr(S^3_{\omega})$ is equivalent to $Gr(S^3_{\omega'})$;
\item $\mathbb{P}^2_{\omega}$ is isomorphic to $\mathbb{P}^2_{\omega'}$;
\item $\mathbb{C}(S^3_{\omega})$ is isomorphic to $\mathbb{C}(S^3_{\omega'})$.
\end{enumerate}
\end{theorem}

\begin{remark}
This theorem clearly shows that point varieties of isomorphic spaces within the family $\mathbb{P}_{\omega}^2$ are the same. There are only two possibilities for the point variety: either $\mathbb{P}^2$ or the triangle formed by the lines $x=0$, $y=0$ and $z=0$ in $\mathbb{P}^2$ with coordinates $(x:y:z)$.  In the first case we are talking about a linear $\mathbb{P}_{\omega}^2$, i.e., $S^3_{\omega}$ is a twisted coordinate ring and thus, by theorem \ref{Mori}, its category of graded modules is equivalent to the one of commutative polynomials. Hence, an isomorphic noncommutative projective space would also be isomorphic to the commutative one and therefore it would have the same point variety. On the other hand, if the point variety is a triangle, then the noncommutative projective space is not isomorphic to the commutative one and thus the point variety has to be preserved under equivalence.
\end{remark}
\begin{remark}
For generic $q$ (in this case, just meaning that $q_{12}:=q_{(123)}(\omega)$ is not a root of unit), the theorem above is also a corollary of theorem \ref{bir class}. In fact, suppose that $\mathbb{P}^{2}_{\omega}$ and $\mathbb{P}^{2}_{\omega'}$ are birationally equivalent. Then there is $A\in GL_{2}(\mathbb{Z})$ such that $q_{12}'= q_{12}^{det(A)}$ and thus, by theorem \ref{graded equivalence}, $Gr(S_{\omega})$ is equivalent to $Gr(S_{\omega'})$.
\end{remark}

The theorem above is not true for higher dimensional $\mathbb{P}_{\omega}^n$'s as the following example shows:
\begin{example}
Let $X=\mathbb{P}^3_{\omega}$ and $X'=\mathbb{P}^3_{\omega'}$ where:
\begin{equation}\nonumber
\omega=(\omega_{ij})_{1\leq i,j\leq 4}=\left(\begin{array}{cccc}1&1&1&1/2\\ 1&1&1&1\\ 1&1&1&1\\ 2&1&1&1\end{array}\right)
\end{equation}
and
\begin{equation}\nonumber
\omega'=(\omega'_{ij})_{1\leq i,j\leq 4}=\left(\begin{array}{cccc}1&1&2&1/2\\ 1&1&1&1\\ 1/2&1&1&1/8\\ 2&1&8&1\end{array}\right).
\end{equation}
We will use theorem \ref{bir class} to check that $X$ and $X'$ are birational. We use the notation, as in the proposition, $q_{ij}:=q_{(1ij)}(\omega)$ and $q_{ij}':=q_{(1ij)}(\omega')$. Consider the matrix
\begin{equation}\nonumber
A=\begin{pmatrix}
a_{22} & a_{23} & a_{24} \\
a_{32} & a_{33} & a_{34} \\
a_{42} & a_{43} & a_{44} \end{pmatrix}
=\begin{pmatrix}
1 & 0 & 0 \\
0 & -1 & 0 \\
-1 & 0 & 1 \end{pmatrix}.
\end{equation}
We now check that this matrix relates $q'_{ij}$'s and $q_{ij}$'s as expected. Since
\begin{equation}\nonumber
q_{23}=1,\ \ \ q_{24}=q_{34}=2\ \ \ {\rm and} \ \ \ q_{23}'=q_{34}'=1/2,\ \ \ q_{24}'=2, 
\end{equation}
by using the fact that $q_{ij}=q_{ji}^{-1}$ we have:
\begin{equation}\nonumber
q_{23}'=q_{23}^{a_{22}a_{33}-a_{32}a_{23}}q_{24}^{a_{22}a_{43}-a_{42}a_{23}}q_{34}^{a_{32}a_{43}-a_{42}a_{33}}= 1.2^0.2^{-1}=1/2
\end{equation}
\begin{equation}\nonumber
q_{24}'=q_{23}^{a_{22}a_{34}-a_{32}a_{24}}q_{24}^{a_{22}a_{44}-a_{42}a_{24}}q_{34}^{a_{32}a_{44}-a_{42}a_{34}}=1.2^1.2^0=2
\end{equation}
\begin{equation}\nonumber
q_{34}'=q_{23}^{a_{23}a_{34}-a_{33}a_{24}}q_{24}^{a_{23}a_{44}-a_{43}a_{24}}q_{34}^{a_{33}a_{44}-a_{43}a_{34}}=1.2^0.2^{-1}=1/2
\end{equation}
as expected, thus proving the birational equivalence of $X$ and $X'$.

However, using corollary \ref{P3points}, we can easily see that their point varieties are not isomorphic: the point variety of $X$ is formed by two hyperplanes and one line while the point variety of $X'$ is formed by six lines. Thus $X$ cannot be isomorphic to $X'$.
\end{example}

One may ask, however, how far theorem \ref{Mori} is from being true for $\mathbb{P}_{\omega}^3$. The previous example shows that (3) does not imply (2) (and thus neither (1)) because some birational equivalences fail to preserve the point variety. The natural question is then: what can we say about birationally equivalent $\mathbb{P}_{\omega}^3$'s with isomorphic point varieties?

\begin{example}
Let $X=\mathbb{P}^3_{\omega}$ and $X'=\mathbb{P}^3_{\omega'}$ where:
\begin{equation}\nonumber
\omega=(\omega_{ij})_{1\leq i,j\leq 4}=\left(\begin{array}{cccc}1&1&1&1/2\\ 1&1&1&1\\ 1&1&1&4\\ 2&1&1/4&1\end{array}\right)
\end{equation}
and
\begin{equation}\nonumber
\omega'=(\omega'_{ij})_{1\leq i,j\leq 4}=\left(\begin{array}{cccc}1&1&1&1/4\\ 1&1&1&1\\ 1&1&1&1/2\\ 4&1&2&1\end{array}\right)
\end{equation}
Then it is easy to see that $X$ and $X'$ are birational since the matrix:
\begin{equation}\nonumber
B=\begin{pmatrix}
b_{22} & b_{23} & b_{24} \\
b_{32} & b_{33} & b_{34} \\
b_{42} & b_{43} & b_{44} \end{pmatrix}
=\begin{pmatrix}
1 & 2 & 0 \\
0 & 1 & 1 \\
0 & 1 & 2 \end{pmatrix}\end{equation}
is relating $q_{(1ij)}(\omega')=:q'_{ij}$ and $q_{(1ij)}(\omega)=:q_{ij}$ as in theorem \ref{bir class} (see previous example for similar computations). It is also clear, by corollary \ref{P3points}, that they have the same point variety (one hyperplane and three lines). However, looking at the $q$-cyclic numbers
\begin{equation}\nonumber
q_{(123)}(\omega)=1,\ \ \ q_{(124)}(\omega)=2,\ \ \ q_{(134)}(\omega)=8,\ \ \  q_{(234)}(\omega)=4\ \ {\rm and} 
\end{equation}
\begin{equation}\nonumber
q_{(123)}(\omega')=1,\ \ \ q_{(124)}(\omega')=4,\ \ \ q_{(134)}(\omega')=2,\ \ \ q_{(234)}(\omega')=1/2 
\end{equation}
and using theorem \ref{graded equivalence} we conclude that $Gr(S_{\omega})$ and $Gr(S_{\omega'})$ are not equivalent, hence showing that, in general,  $(3)\Longrightarrow (1)$ of theorem \ref{Mori} fails for $\mathbb{P}^3$ even with the additional condition of isomorphic point varieties.
\end{example}

The question of whether $(2)\Rightarrow (1)$ of theorem \ref{Mori} holds for higher dimensions remains without an answer. We, however, conjecture it to be true.

\begin{conjecture}
If $\mathbb{P}^{n-1}_{\omega}\cong\mathbb{P}^{n-1}_{\omega'}$ then $Gr(S_{\omega})\cong Gr(S_{\omega'})$.
\end{conjecture}
\end{section}

\end{document}